\documentclass[reqno,11pt]{amsart}

\usepackage{amssymb,amsmath,amsthm,amsfonts,mathtools,bbm,mathrsfs,enumitem,graphics,float,multirow,enumitem,cases}
\usepackage[normalem]{ulem}
\usepackage{hyperref}
\usepackage{comment}
\usepackage{hyperref}
\hypersetup{
    colorlinks,
    linkcolor={blue!80!black},
    citecolor={red!80!black},
    urlcolor={red!80!black}
}
\usepackage{xargs}                      
\usepackage{xcolor}  
\usepackage[colorinlistoftodos,prependcaption,textsize=tiny]{todonotes}
\newcommandx{\unsure}[2][1=]{\todo[linecolor=red,backgroundcolor=red!25,bordercolor=red,#1]{#2}}
\newcommandx{\change}[2][1=]{\todo[linecolor=blue,backgroundcolor=orange!25,bordercolor=blue,#1]{#2}}
\newcommandx{\info}[2][1=]{\todo[linecolor=OliveGreen,backgroundcolor=OliveGreen!25,bordercolor=OliveGreen,#1]{#2}}

\usepackage[top=1.6in,bottom=1.5in,left=1.3in,right=1.3in]{geometry}
\allowdisplaybreaks
\usepackage{cellspace}
\newcommand\redout{\bgroup\markoverwith{\textcolor{red}{\rule[0.5ex]{2pt}{0.8pt}}}\ULon}

\newtheorem{theorem}{Theorem}
\newtheorem{lemma}[theorem]{Lemma}
\newtheorem{claim}[theorem]{Claim}

\usepackage{tikz}
\usetikzlibrary{calc}

\tikzset{font= }

\newcommand\nc\newcommand
\nc\bfa{{\boldsymbol a}}\nc\bfA{{\boldsymbol A}}\nc\cA{{\mathscr A}}
\nc\bfb{{\boldsymbol b}}\nc\bfB{{\boldsymbol B}}\nc\cB{{\mathscr B}}
\nc\bfc{{\boldsymbol c}}\nc\bfC{{\boldsymbol C}}\nc\cC{{\mathscr C}}
\nc\bfd{{\boldsymbol d}}\nc\bfD{{\boldsymbol D}}\nc\cD{{\mathscr D}}
\nc\bfe{{\boldsymbol e}}\nc\bfE{{\boldsymbol E}}\nc\cE{{\mathscr E}}
\nc\bff{{\boldsymbol f}}\nc\bfF{{\boldsymbol F}}\nc\cF{{\mathscr F}}
\nc\bfg{{\boldsymbol g}}\nc\bfG{{\boldsymbol G}}\nc\cG{{\mathscr G}}
\nc\bfh{{\boldsymbol h}}\nc\bfH{{\boldsymbol H}}\nc\cH{{\mathscr H}}\nc\fH{{\mathfrak H}}
\nc\bfi{{\boldsymbol i}}\nc\bfI{{\boldsymbol I}}\nc\cI{{\mathcal I}}
\nc\bfj{{\boldsymbol j}}\nc\bfJ{{\boldsymbol J}}\nc\cJ{{\mathscr J}}
\nc\bfk{{\boldsymbol k}}\nc\bfK{{\boldsymbol K}}\nc\cK{{\mathscr K}}
\nc\bfl{{\boldsymbol l}}\nc\bfL{{\boldsymbol L}}\nc\cL{{\mathscr L}}
\nc\bfm{{\boldsymbol m}}\nc\bfM{{\boldsymbol M}}\nc\cM{{\mathscr M}}
\nc\bfn{{\boldsymbol n}}\nc\bfN{{\boldsymbol N}}\nc\sN{{\mathscr N}}
\nc\bfo{{\boldsymbol o}}\nc\bfO{{\boldsymbol O}}\nc\cO{{\mathscr O}}
\nc\bfp{{\boldsymbol p}}\nc\bfP{{\boldsymbol P}}\nc\cP{{\mathscr P}}
\nc\bfq{{\boldsymbol q}}\nc\bfQ{{\boldsymbol Q}}\nc\cQ{{\mathscr Q}}
\nc\bfr{{\boldsymbol r}}\nc\bfR{{\boldsymbol R}}\nc\cR{{\mathscr R}}
\nc\bfs{{\boldsymbol s}}\nc\bfS{{\boldsymbol S}}\nc\cS{{\mathscr S}}
\nc\bft{{\boldsymbol t}}\nc\bfT{{\boldsymbol T}}\nc\cT{{\mathscr T}}
\nc\bfu{{\boldsymbol u}}\nc\bfU{{\boldsymbol U}}\nc\cU{{\mathscr U}}
\nc\bfv{{\boldsymbol v}}\nc\bfV{{\boldsymbol V}}\nc\cV{{\mathscr V}}
\nc\bfw{{\boldsymbol w}}\nc\bfW{{\boldsymbol W}}\nc\cW{{\mathscr W}}
\nc\bfx{{\boldsymbol x}}\nc\bfX{{\boldsymbol X}}\nc\cX{{\mathscr X}}
\nc\bfy{{\boldsymbol y}}\nc\bfY{{\boldsymbol Y}}\nc\cY{{\mathscr Y}}
\nc\bfz{{\boldsymbol z}}\nc\bfZ{{\boldsymbol Z}}\nc\cZ{{\mathscr Z}}
\nc\pp{\mathbb{P}}
\nc\ee{\mathbb{E} }

\renewcommand{\le}{\leqslant}
\renewcommand{\leq}{\leqslant}
\renewcommand{\ge}{\geqslant}
\renewcommand{\geq}{\geqslant}

\DeclareMathOperator{\Var}{Var}
\nc{\Cay}{{\sf Cay}}

\nc{\ff}{{\mathbb F}}

\title[]{Bootstrap percolation on the hypercube}
\author[]{Fengxing Zhu}\thanks{Institute for Systems Research and Department of ECE, University of Maryland, College Park, MD 20742, USA, fengxing@terpmail.umd.edu. Supported in part by NSF grant CCF 2330909.}
\date{}

\begin{document}
\begin{abstract}
    In this paper we obtain the critical probability  $p_c(Q_{n},r)$ for bootstrap percolation with the infection threshold $r=f(n)n^a$ on the $n$-dimensional hypercube $Q_{n}$, where $0<a<1$ and $f(n)=n^{o(1)}$. More precisely, we obtain the exact first-order and second-order terms of $p_c(Q_{n},r)$ by extending the main theorem of Balogh, Bollob{\'a}s, and Morris (2009). 
\end{abstract}
\maketitle

\section{Introduction}
Bootstrap percolation was first introduced by Chalupa, Leith, and Reich \cite{Chalupa_1979}. Given an undirected graph $G=(V,E)$ and an integer $r\geq 1$, the $r$-neighbor bootstrap percolation process is a discrete-time infection model in which vertices are classified as either infected or healthy. The infection is irreversible, meaning that infected vertices remain infected forever. Let $A_0\subseteq V$ be the set of initially infected vertices, and let $A_i$ denote the set of infected vertices up to step $i$. At each subsequent step, a healthy vertex becomes infected if at least $r$ of its neighbors were infected in the preceding steps. More precisely, for every $i\geq 1$,
$$
A_i=A_{i-1}\cup \{v\in V:|\mathcal{N}(v)\cap A_{i-1}|\geq r \},
$$
where $\mathcal{N}(v)$ denotes the set of neighbors of $v$ in $G$. The integer $r$ is referred to as the infection threshold. An initial configuration $A_0$ is called a contagious set if the infection eventually reaches every vertex of $G$. In that case we say that the process percolates.

A natural probabilistic model based on this process, known as random bootstrap percolation, arises when each vertex is initially infected independently with probability $p$. One of the main objectives in studying this random process is to determine the critical probability. Specifically, the critical probability is defined as
$$
p_c(G,r):=\sup \{p\in(0,1):\mathbb{P}_p(A\text{ percolates on }G)\leq\frac{1}{2} \},
$$
where $A$ is the random set of initially infected vertices and $\mathbb{P}_p$ denotes the corresponding probability measure.

 The study of bootstrap percolation on the $d$-dimensional box $[n]^d=\{1,2,\cdots,n\}^d$ dates back to the work of Aizenman and Lebowitz \cite{Aizenman_1988}, who showed that, for $d\geq r=2$, the critical probability satisfies
$$
p_c([n]^d,2)=\Theta\left((\log n)^{1-d}\right).
$$
Building on the sharp threshold results of Friedgut \cite{Friedgut}, Balogh and Bollob{\'a}s \cite{BALOGH2003305} established a sharp transition for bootstrap percolation on $d$-dimensional boxes. A major breakthrough was subsequently achieved by Holroyd \cite{Holroyd2002SharpMT}, who determined the exact leading-order asymptotics of the critical probability in the two-dimensional case $d=r=2$. Further refinements by Gravner and Holroyd \cite{holroyd_2008}, followed by Hartarsky and Morris \cite{Robert}, established the order of the second-order term, yielding
$$
p_c([n]^2,2)=\frac{\pi}{18\log n}-\frac{\Theta(1)}{(\log n)^{3/2}}.
$$

For higher dimensions, Cerf and Cirillo \cite{Cirillo} proved that when $d=r=3$, the critical probability is of order $1/\log\log n$. This was extended by Cerf and Manzo \cite{CERF200269}, who established that, for all $d\geq r\geq 3$,
$$
p_c([n]^d,r)=\Theta\left(\left(\log^{(r-1)}n\right)^{r-d-1}\right),
$$
where $\log^{(r-1)}n$ denotes the $(r-1)$-fold iterated logarithm of $n$. Subsequently, Balogh, Bollob{\'a}s, and Morris \cite{Balogh2} obtained the exact first-order term of the critical probability for $d=r=3$. Their results were further generalized by Balogh, Bollob{\'a}s, Duminil-Copin, and Morris \cite{Balogh}, who established the sharp asymptotic threshold for the $r$-neighbor bootstrap percolation process on $[n]^d$ for every $d\geq r\geq 2$. In particular, they proved that there exists an explicit constant $\lambda(d,r)>0$ such that
$$
p_c([n]^d,r)=
\left(\frac{\lambda(d,r)+o(1)}{\log^{(r-1)}n}\right)^{d-r+1}.
$$

In addition to the extensive research on random bootstrap percolation in $d$-dimensional boxes, the hypercube has also received considerable attention. Let $Q_n$ denote the $n$-dimensional binary hypercube, whose vertex set is $\{0,1\}^n$, with two vertices being adjacent if their Hamming distance is one. For the $2$-neighbor bootstrap percolation process on $Q_n$, Balogh and Bollob{\'a}s \cite{balogh_bollobas_2006} proved that there exist positive constants $c_1$ and $c_2$ such that
$$
\frac{c_1}{n^2}2^{-2\sqrt{n}}
\leq p_c(Q_n,2)
\leq \frac{c_2}{n^2}2^{-2\sqrt{n}}.
$$
This result was subsequently refined by Balogh, Bollob{\'a}s, and Morris \cite{balogh_bollobas_morris_2010}, who obtained the exact first-order term and good estimates for the second-order term of the critical probability. In particular, they showed that there exists an explicit constant $\lambda>0$ for which
$$
\frac{16\lambda}{n^2}
\left(1+\frac{\log n}{\sqrt{n}}\right)2^{-2\sqrt{n}}
\leq p_c(Q_n,2)
\leq
\frac{16\lambda}{n^2}
\left(1+\frac{5(\log n)^2}{\sqrt{n}}\right)2^{-2\sqrt{n}}.
$$

Another important regime of bootstrap percolation on the hypercube arises when the infection threshold is $r=n/2$. In this setting, Balogh, Bollob{\'a}s, and Morris \cite{balogh_bollobas_morris_2009} investigated the critical probability on $Q_n$ and established the following bounds:
$$
\frac{1}{2}
-\frac{1}{2}\sqrt{\frac{\log n}{n}}
-\frac{2\log\log n}{\sqrt{n\log n}}
\leq p_c\left(Q_n,\frac{n}{2}\right)\
\leq
\frac{1}{2}
-\frac{1}{2}\sqrt{\frac{\log n}{n}}
+\frac{\log\log n}{2\sqrt{n\log n}}
+o\left(\frac{\log\log n}{2\sqrt{n\log n}}\right).
$$

Recently, Collares, Erde, Geisler, and Kang \cite{collares2024universalbehaviourmajoritybootstrap} extended the study of random majority bootstrap percolation from the hypercube to a broad class of high-dimensional geometric graphs, including Hamming graphs. They established universal behavior in the location and width of the critical window, showing that these graph families exhibit similar quantitative behavior under suitable structural assumptions. In a subsequent work \cite{COLLARES2026104347}, the same authors further developed this framework by extending their results to certain high-dimensional graphs of super-exponential order, including the permutahedron. They also obtained improved bounds on the critical window, particularly for the hypercube, and investigated majority bootstrap percolation on irregular graphs. 

\emph{Overview of the results.} Although majority bootstrap percolation on the hypercube has been extensively studied in \cite{balogh_bollobas_morris_2009} and subsequently extended to more general classes of graphs in \cite{collares2024universalbehaviourmajoritybootstrap,COLLARES2026104347}, these works primarily focus on the setting where the infection threshold is half the degree of each vertex. In contrast, the regime in which the infection threshold scales as $r=n^a$ for $0<a<1$ on the hypercube has remained unexplored prior to the present work. Related questions were raised by Collares, Erde, Geisler, and Kang in Section 9 of \cite{collares2024universalbehaviourmajoritybootstrap}. In particular, they asked which classes of $d$-regular graphs have a critical probability of order $d^{-a}$ for some $0<a<1$, and how the critical probability behaves when the infection threshold is set to $r=\sqrt{d}$.

Our work provides partial answers to these questions by determining the first-order and second-order terms of the critical probability for bootstrap percolation on the hypercube with infection threshold $r=f(n)n^a$, where $0<a<1$ and $f(n)=n^{o(1)}$. In particular, we prove the following theorem. 

\begin{theorem} \label{main theorem}
Let $r=f(n)n^a$ where $0<a<1$ and $f(n)=n^{o(1)}$. For every $\epsilon >0$ there exists an $n_0(\epsilon)$ such that for all $n \geq n_0(\epsilon)$, we have 
$$ \frac{r}{n}-\left(\sqrt{2-a}+\epsilon \right)\frac{\sqrt{r \log n}}{n} \leq p_c(Q_n,r) \leq \frac{r}{n}-\left(\sqrt{2-a}-\epsilon \right)\frac{\sqrt{r \log n}}{n}.$$
\end{theorem}
\emph{Our Contributions.}
Our first contribution lies in the proof of the upper bound in Theorem~\ref{main theorem}. Specifically, we establish that $\mathbb{P}(x\notin A_2)\leq n^{-\gamma}$ by combining Petrov's approximation in Lemma~\ref{lemma: Petrov}, Gordon's ratio in Lemma~\ref{lemma:Gordon}, and McKay's estimate in Lemma~\ref{Lemma:McKay} to obtain accurate estimates of binomial tail probabilities. Together with Chebyshev's inequality, these estimates allow us to show that, for any fixed vertex $x$, with probability at least $1-n^{-\gamma}$, the number of its neighbors that become newly infected in the first step is sufficient to bridge the gap between the infection threshold and the number of its initially infected neighbors.

Our second contribution lies in the proof of the lower bound in Theorem~\ref{main theorem}. Inspired by the approach of \cite{balogh_bollobas_morris_2009}, we introduce a more generous process in which the infection thresholds are strictly smaller than $r$ during the first few steps. The number of these steps depends on the value of $a$. The none uniform gaps between these infection thresholds are carefully chosen to achieve sufficiently strong exponential decay to compensate for the union bounds.

\emph{Our Proof Techniques.} Our main proof strategy follows that of \cite{balogh_bollobas_morris_2009}. The underlying heuristic is as follows. When the initial infection probability is below the critical probability, the bootstrap percolation process stops after only a few steps with high probability and when it is above the critical probability, all vertices become infected within a few steps with high probability.

The heuristic for the first-order term is that the critical probability is primarily determined by the local structure of the graph, with its leading term given by $\frac{r}{|\mathcal{N}(x)|}$. The heuristics for the second-order term is as follows. Let $r=n^a$ and write $ p=\frac{r-\lambda}{n}.$ For a fixed vertex $x$, let $X=|\mathcal{N}(x) \cap A_0| \sim \text{Binomial}(n,p)$. Then $\mathbb{E}(X)=np=r-\lambda$ and $\text{Var}(X)=np(1-p)=(r-\lambda)(1-o(1))$

Hence the probability that $x$ becomes newly infected in the first round is 
$$\mathbb{P}(X \geq r, x \notin A_0)=(1-p)\mathbb{P}(X \geq r).$$

If $ \lambda\gg\sqrt r$, $\lambda=o(r)$, and $\frac{\lambda}{\sqrt{r}}=o(n^{1/6})$, then by Lemmas \ref{lemma: Petrov} and \ref{lemma:Gordon}, we have 
$$
\mathbb{P}(X \geq r) =(1+o(1)) \frac{\sqrt r}{\sqrt{2\pi}\lambda} \exp\left(-\frac{\lambda^2}{2r}\right).$$
Now consider a typical healthy vertex $y$ and let $x \in \mathcal{N}(y)$. Initially $y$ has about $r-\lambda$ infected neighbors, so its deficit from the threshold is about $\lambda.$ Among its $n$ neighbors, roughly $n\mathbb{P}(X \geq r, x \notin A_0)$ new vertices become infected in the first round. Therefore the natural transition between all vertices being infected and occurs when $n\mathbb{P}(X \geq r, x \notin A_0) \approx \lambda$. Thus 
$$
n\frac{\sqrt r}{\sqrt{2\pi}\lambda}
e^{-\lambda^2/(2r)} \approx \lambda.
$$
Let $z=\frac{\lambda}{\sqrt r}$. Then we have $ \frac{n}{\sqrt{2\pi}z} e^{-z^2/2} \approx z\sqrt r,$ and so $e^{-z^2/2} \approx z^2\frac{\sqrt r}{n}$. Taking logarithms,
$$
\frac{z^2}{2}
=
\log\frac{n}{\sqrt r}
-2\log z.
$$
Since $r=n^a$, we have $ \log\frac{n}{\sqrt r}= \left(1-\frac {a}{2}\right)\log n.$ Therefore, $$
z^2
\approx (2-a)\log n.
$$
Consequently,
$$
\lambda \approx \left(\sqrt{2-a}\right) \sqrt{r\log n}.
$$

To establish the upper bound on the critical probability, we first show that the probability that a fixed vertex $x$ remains uninfected after the second step is at most $n^{-\gamma}$ for some constant $\gamma>0$. More precisely, we analyze $|\mathcal{N}(x)\cap (A_1\setminus A_0)|$ and show that, with probability $1-n^{-\gamma}$, this quantity is large enough to bridge the gap between the infection threshold and the expected number of initially infected neighbors of $x$. We then show that the probability that $x$ remains uninfected after the fifth step is exponentially small in $n$, in the order of $\exp(-n \log n)$, allowing us to apply a union bound over all vertices to complete the proof.

The main difference between our proof of the upper bound and that of \cite{balogh_bollobas_morris_2009} lies in the strength of the probability estimates obtained during the early stages of the process. Balogh, Bollob{\'a}s, and Morris first established that $\mathbb{P}(x\in A_2)\geq \frac{3}{4}$ and subsequently showed that $\mathbb{P}(x\notin A_5)=\exp(-\Theta(n))$. To apply a union bound over all vertices, they proved an additional lemma showing that $\mathbb{P}(x\notin A_{11})\leq \exp(-\Theta(n^2))$. In contrast, by combining Petrov's approximation, Gordon's ratio, and McKay's estimate, we obtain the stronger bound $\mathbb{P}(x\notin A_2)\leq n^{-\gamma}$. This sharper estimate allows us to adapt the arguments in \cite{balogh_bollobas_morris_2009} to establish that $\mathbb{P}(x\notin A_5)\leq \exp(-\Theta(n\log n))$, which is sufficiently small to apply a union bound over all $2^n$ vertices and complete the proof.

To prove the lower bound in Theorem~\ref{main theorem}, we introduce a more generous process in which the infection thresholds are strictly smaller than $r$ during the first few steps. We then show that, with high probability, this modified process stops after only a few steps, leaving many vertices uninfected. Since the modified process dominates the original process, it follows that the original process also fails to percolate with high probability.

The main difference between our proof and that of \cite{balogh_bollobas_morris_2009} lies in the construction of the more generous process. Let $B_i$ denote the set of vertices infected by step $i$ in this modified process. Specifically, we consider the first $k$ steps, where $k>\frac{2}{a}$, and choose the infection thresholds non-uniformly during these steps. The gaps between these thresholds play a crucial role in our argument and must be carefully chosen to ensure that Bernstein's inequality yields sufficiently strong exponential decay to overcome two union bounds. The first is taken over all possible witness chains, which, roughly speaking, are sequences of sets from $B_i\setminus B_{i-1}$, for $i\in[k]$, that witness the event $x\in B_{k+1}\setminus B_k$. The second is taken over all $2^n$ vertices of the hypercube.

\section{Preliminaries and Tools}

We  need the following multiplicative form of the Chernoff's bound. 
\begin{lemma} \label{lemma: chernoff}
Suppose $X_1, ..., X_n$ are independent random variables taking values in $\{0, 1\}$. Let $X$ denote their sum and define $\mu: = \mathbb{E}[X]$. Then for any $\delta > 0 $
$$\mathbb{P}(X \leq (1-\delta)\mu)\leq \exp\left(\frac{-\delta^2 \mu}{2} \right).$$
\end{lemma}
We also use the following lemma which was proved in \cite{balogh2007majority}.

\begin{lemma} \label{lemma: graph partition for hypercube}
 Let $n,k \in \mathbb{N}$, and $x \in V(Q_n)$. Then there exists a partition 
 $$S(x,k):=\{y \in V:d_H(y,x)=k\}=B_1 \cup \cdots \cup B_m$$
 of $S(x,k)$ into $m \leq k \binom{n}{k-1} \leq 2n^{k-1}$ sets, such that if $x,y \in B_j$ for some $j$, then $d_H(x,y) \geq 2k$. 
\end{lemma}

We need to approximate the binomial distribution very accurately. The following three lemmas will play a crucial role in our analysis. 

\begin{lemma} (Petrov's approximation) Chapter: Probabilities of Large Deviations in \cite{Petrov1975} \label{lemma: Petrov}

 Let $V_{N}^2=\sum_{i=1}^{N} \Var(X_{_i})$ and $S_{N}=\sum_{i=1}^{N} X_{i}$, where the centered random variables $X_{i}$  are independent with a common distribution. Assume that there exists a constant $h>0$, independent of $i$, such that 
$$\sup_{i}\mathbb{E}e^{tX_{i}}<\infty,$$
in the interval $|t| \leq h$. 
Then we have for $x \geq 0$ and $x=o(\sqrt{N})$, 
$$\mathbb{P}(S_{N} \geq x V_{N})=\left(1-\Phi(x) \right)  \exp\left(\frac{x^3}{\sqrt{N}} \lambda \left(\frac{x}{\sqrt{N}} \right) \right) \left(1+O\left(\frac{1+x}{\sqrt{N}} \right) \right),$$
where $\Phi(x)$ denotes the cumulative distribution function of a standard normal random variable,i.e.,
$$\Phi(x) =\frac{1}{\sqrt{2\pi}}\int_{-\infty}^{x} e^{-t^2/2}dt,$$
and 
$\lambda(t)=\sum_{k=0}^{\infty} a_kt^k$ is a power series with coefficients depending on the cumulants of the random variable $X_i$ which converges for sufficiently small values of $|t|$. 
\end{lemma}

\begin{lemma} (Gordon's ratio) \cite{Gordon1941} \label{lemma:Gordon}

For $x >0$, we have 
$$\frac{x}{1+x^2}\phi(x) \leq 1- \Phi(x) \leq \frac{\phi(x)}{x},$$
where $\phi(x)=\frac{1}{\sqrt{2\pi}}e^{-x^2/2}$. 
    
\end{lemma}

\begin{lemma} (McKay's estimate)\cite{Brendan} \label{Lemma:McKay}

Let $0< p <1$ , $n \geq 1$, and $pn \leq k \leq n$. Define $x=\frac{k-pn}{\sigma}$ and $\sigma^2=\text{Var}(\text{Binomial}(n,p))$. Then 
$$B(k;n,p)=\sigma b(k-1;n-1,p)Y(x) \exp(E(k;n,p)/\sigma),$$
where 
$$B(k;n,p)=\sum_{j=k}^nb(j;n,p), \quad
b(k;n,p)=\binom{n}{k}p^k(1-p)^{n-k},$$
$$Y(x)=\frac{1-\Phi(x)}{\phi(x)}, \quad 0 \leq E(k;n,p) \leq \min\{\frac{1}{x},\sqrt{\frac{\pi}{8}}\}.$$
\end{lemma}

The Chebyshev's inequality is also important for us. 
\begin{lemma} (Chebyshev's inequality) \label{lemma:Chebyshev}

  Let $X$ (integrable) be a random variable with finite non-zero variance $\sigma^2$ and finite expected value $\mu$. Then for any real number $a >0$, 
$$\mathbb{P}(|x-\mu| \geq a) \leq \frac{\sigma^2}{a^2}.$$  
\end{lemma}

The following lemma is important for us to prove the lower bound of Theorem ~\ref{main theorem}.
\begin{lemma} Bernstein's Inequality \label{lemma: Bernstein}

Let $X_1,X_2,\cdots,X_n$ be independent zero-mean random variables. Suppose that $|X_i| \leq M$ almost surely for all $i \in [n]$. Then for all $t >0$, 
$$\mathbb{P}\left(\sum_{i=1}^{n}X_i \geq t\right) \leq \exp\left(-\frac{\frac{1}{2}t}{\sum_{i=1}^n \mathbb{E}(X_i^2)+\frac{1}{3}Mt} \right).$$

\end{lemma}

\section{Upper Bound}
For the ease of notations, from now on by $c$ we denote a constant which does not depend on $n$. We will use $c$ as generic notations, where the specific forms may be different in different expressions. 

In this section, we establish the upper bound in Theorem~\ref{main theorem}. Our proof strategy is similar to that of \cite{balogh_bollobas_morris_2009}. We first show that $\mathbb{P}(x\notin A_2)\leq n^{-\gamma}$ for some constant $\gamma>0$. We then analyze the evolution of the process over a few additional steps to obtain the stronger bound $\mathbb{P}(x\notin A_5)\leq \exp(-\Theta(n\log n))$. This estimate allows us to apply a union bound over all $2^n$ vertices of the hypercube to complete the proof.

\begin{lemma}
Let the initial infection probability $p$ be $\frac{r}{n}-\left(\sqrt{2-a}-\epsilon \right)\frac{\sqrt{r \log n}}{n}$ and $x$ be a vertex in $V(Q_n)$. Then we have 
$$\mathbb{P}(x \notin A_2) \leq n ^{-\gamma},$$
where $\gamma $ is a positive constant. 
\end{lemma}
\begin{proof}
The main idea in this proof is as follows. Let $x \in V(Q_n)$ and without loss of generality we can assume that $x=[0]^n$. We prove there are enough number of newly infected vertices in $\mathcal{N}(x)$ in the first step to bridge the gap between $r$ and $np$. 

First we will show that $\sum_{y \in \mathcal{N}(x)} \mathbb{I}_{\{|\mathcal{N}(y) \in A_0| \geq r\}}$ is concentrated on its mean by using Chebyshev's inequality. Thus we need to analyze its mean and variance.

We can assume that $x \notin A_0$ because otherwise there is nothing to prove. Let $y,z \in \mathcal{N}(x)$ and $y \neq z$. Since $\mathcal{N}(x) \subset S(x,1)$, $y,z \in S(x,1)$. Now consider $\mathcal{N}(y)$ and $\mathcal{N}(z)$. we have 
$$\mathcal{N}(y) \cap \mathcal{N}(z)=\{x,w\}, $$
where $w \in S(x,2)$. 

Note that since $x \notin A_0$, $|\mathcal{N}(y) \cap A_0| \sim \text{Binomial}(n-1,p)$. Let $X:=|\mathcal{N}(y) \cap A_0|$. We have 
$\mu=\mathbb{E}(X)=(n-1)p$ and $\sigma^2=\text{Var}(X)=(n-1)p(1-p)$. 

Let $\lambda:=(\sqrt{2-a}-\epsilon)\sqrt{r \log n}$. Since $np=r-\lambda$, then $\mu=r-\lambda-p$ and thus $r-\mu=\lambda+p$. Since $r=o(n)$, $\sigma^2=(1+o(1))r$. Let $C=\sqrt{2-a}-\epsilon $. Define $$b:= \frac{r-\mu}{\sigma}=C \sqrt{\log n}(1+o(1)).$$ 

By Lemmas \ref{lemma: Petrov} and \ref{lemma:Gordon}, we have 
$$\mathbb{P}(|\mathcal{N}(y) \cap A_0| \geq r)=\frac{1+o(1)}{\sqrt{2\pi}b}\exp\left(\frac{-b^2}{2}\right)=n^{\frac{-C^2}{2}+o(1)}$$
Let $I_y:=\mathbb{I}_{\{|\mathcal{N}(y) \cap A_0| \geq r\}}$ and $I_z:=\mathbb{I}_{\{|\mathcal{N}(z) \cap A_0| \geq r\}}$. Now we need to estimate $\text{Cov}(I_y,I_z)$. 

We have $|\mathcal{N}(y)\cap A_0|= V_{w}+U_y$ and  $|\mathcal{N}(z)\cap A_0|= V_{w}+U_z$, where $V_w \sim \text{Bernoulli}(p)$, $U_y \sim \text{Binomial}(n-2,p)$ and $U_z \sim \text{Binomial}(n-2,p)$ and $V_w$, $U_y$, $U_z$ are independent.

Define $q_0:=\mathbb{P}(U_y\ge r), $ and $q_1:=\mathbb{P}(U_y\ge r-1)$. 

Now condition on $V_w$. If $V_w=0$, then
$$
I_y=1\iff U_y\ge r,
$$
and similarly for $z$. Hence
$$
\mathbb{P}(I_y=I_z=1| V_w=0)=q_0^2,
$$
because $U_i$ and $U_j$ are independent.
If $V_w=1$, then
$$
I_y=1\iff U_y\ge r-1,
$$
so
$$
\mathbb{P}(I_y=I_z=1|V_w=1)=q_1^2.
$$
Therefore
$$
\mathbb E[I_yI_z]
=
(1-p)q_0^2+p q_1^2.
$$
On the other hand,
$$
\mathbb E I_y
=
(1-p)q_0+p q_1.
$$
The same is true for $I_z$, so
$$
\mathbb E I_y\,\mathbb E I_z
=
\big((1-p)q_0+p q_1\big)^2.
$$
Thus
\begin{align*}
\operatorname{Cov}(I_y,I_z)
&=\mathbb E[I_yI_z]-\mathbb E I_y\,\mathbb E I_y\\
&=(1-p)q_0^2+p q_1^2
-\big((1-p)q_0+p q_1\big)^2,\\
&=p(1-p)(q_1-q_0)^2.
\end{align*}

\begin{claim} \label{claim: Cov}
$$\text{Cov}(I_y,I_z)= O\left(\frac{p\tau^2\lambda^2}{r^2} \right) $$
where $\tau=\mathbb{P}(|\mathcal{N}(y) \cap A_0|\geq r)$. 
\end{claim}
\begin{proof}
Note that it is sufficient to prove that 
\begin{equation}
 q_1-q_0=\mathbb{P}(U_y=r-1)=O\left(\frac{\tau \lambda}{r}\right).   \label{eq: q_1-q_0}
\end{equation} 

We will use Lemmas ~\ref{lemma:Gordon} and \ref{Lemma:McKay} to show (\ref{eq: q_1-q_0}). More precisely, we show  
$$\mathbb{P}(U_y=r-1)=(1+o(1))\frac{\lambda}{r} \mathbb{P}(U_y \geq r-1).$$
Let $ U\sim\text{Binomial}(N,p),$ and write $b(k)=P(U=k), B(k)=P(U\ge k).$
We have the exact identity $
b(k-1;N-1,p)=\frac{k}{Np}\,b(k;N,p).
$
Therefore, Lemma ~\ref{Lemma:McKay} gives
$$
B(k) =\sigma\,\frac{k}{Np}\, b(k)\,Y(x)e^{E/\sigma}.$$
Hence
$$\frac{b(k)}{B(k)} = \frac{Np}{\sigma k} \frac1{Y(x)} e^{-E/\sigma}.$$
Now in our setting,
$ k=r-1,Np=r-\lambda+o(1), \sigma^2=(1+o(1))r,$
and $ k-Np=(1+o(1))\lambda.$
Thus
$$
x=\frac{k-Np}{\sigma}
=(1+o(1))\frac{\lambda}{\sqrt r}
\to\infty.
$$

Then Lemma ~\ref{lemma:Gordon} gives
$$
Y(x)=\frac{1-\Phi(x)}{\phi(x)}
=
\frac{1+o(1)}{x},
$$
so
$$
\frac1{Y(x)}=(1+o(1))x.
$$
Also,
$
0\le \frac E\sigma
\le
\frac{O(1)}{\sigma}
=o(1),
$
because \(\sigma\sim\sqrt r\to\infty\). Therefore
$$
e^{-E/\sigma}=1+o(1).
$$
By Lemma ~\ref{Lemma:McKay},
$$
\frac{P(U=k)}{P(U\ge k)}
=
(1+o(1))
\frac{Np}{k}\frac{x}{\sigma}.
$$
Since $
\frac{Np}{k}=1+o(1),
$
we obtain
$$
\frac{P(U=k)}{P(U\ge k)}
=
(1+o(1))\frac{x}{\sigma}=(1+o(1))\frac{k-Np}{\sigma^2}
=
(1+o(1))\frac{\lambda}{r}.
$$
Finally, note that $\tau =(1-o(1)) \mathbb{P}(U_y \geq r-1)$. 
\end{proof}
Note that $\frac{p\lambda^2}{r^2}=(1+o(1))\frac{C^2\log n}{n}$. Therefore,
$$\text{Var}\left(\sum_{y \in \mathcal{N}(x)}I_y \right)= \sum_{y \neq z}\text{Cov}(I_y,I_z)+\sum_{y}\text{Var}(I_y) \leq O(n\tau^2\log n)+n\tau$$

Let $M=\sum_{y \in \mathcal{N}(x)}I_y$. Note that $\mathbb{E}(M)=\tau n$. Then by Chebyshev's inequality in Lemma \ref{lemma:Chebyshev}, we have 
$$\mathbb{P}\left(M \leq \frac{n\tau}{2} \right) \leq \mathbb{P} \left(|M-\mathbb{E}(M)| \geq \frac{n\tau}{2} \right) \leq \frac{\text{Var}(M)}{(\frac{1}{2} n\tau)^2} \leq O\left(\frac{\log n}{n}\right)+O\left(\frac{1}{n\tau}\right).$$

Now we will show that the number of newly infected vertices in the first step in the neighbor of the vertex $x$ is a good fraction of $\mathbb{E}(M)$. 

The events defining $M$ depend only on $A_0\cap S(x,2)$ since we assume that $x \notin A_0$ and whether $y\in A_0$ depends on $A_0\cap S(x,1)$. They are independent.
Let
$$
H=
|\{y \in \mathcal{N}(x):|\mathcal{N}(y) \cap A_0| \geq r, y\notin A_0\}|.
$$
Conditional on $M=m$, $H\sim\text{Binomial}(m,1-p)$.

Since $p=o(1)$,
$$\mathbb{P}(H\leq \frac{1}{2}m) \leq \exp(-cm),$$
where $c>0$. 

Hence 
$$
\mathbb{P}(H\ge \frac{1}{4}n\tau)\geq 1-n^{-\gamma_0},
$$
where $\gamma_0 >0$. Every one of these $H$ vertices belongs to $A_1\setminus A_0$.

Now we will show these newly infected vertices in the first step is large enough to bridge the gap between $r$ and $np$. 

Note that $n\tau=n^{1-\frac{C^2}{2}+o(1)}$. Let $\beta=1-\frac{C^2}{2}+o(1)$ and choose $\eta>0$ such that $ \frac{a}{2}+\eta<\beta.$
Set $D=n^{a/2+\eta}$ and thus $H \gg D$ with probability at least $1-n^{-\gamma_0}$.

We have $\mathbb{E}|\mathcal{N}(x)\cap A_0|=r-\lambda$.  Since
$ D\gg\lambda $
and $ \frac{D^2}{r}=n^{2\eta},$ Lemma ~\ref{lemma: chernoff} gives
$$
\mathbb{P}(|\mathcal{N}(x) \cap A_0|<r-D)
\le
\exp\{-\Theta(n^{2\eta})\}.
$$
Therefore, 
\begin{align*}
\mathbb{P}(\{|\mathcal{N}(x) \cap A_0|\ge r-D\} \cap \{H \ge D\})^c & \leq \mathbb{P}(|\mathcal{N}(x) \cap A_0|\le r-D) + \mathbb{P}(H \leq D) \\ 
& \leq \exp(-\Theta(n^{2\eta}))+ n^{-\gamma_0}
\end{align*}

Thus we have 
$$\mathbb{P}(x\notin A_2)\le n^{-\gamma}. \qedhere$$

\end{proof}

We need the following lemma to establish that $\mathbb{P}(x \notin A_5)$ is small enough to apply a union bound over all $2^n$ vertices. 
\begin{lemma}
Assume that $x \in V(Q_n)$  and $\mathbb{P}(x \notin A_2) \leq n^{-\gamma}$, where $\gamma$ is a positive constant. Then we have 
$$\mathbb{P}(x \notin A_5) \leq \exp(-cn \log n),$$
where $c$ is a positive constant.
\end{lemma}
\begin{proof}
Suppose that $x \notin A_5$. Then $x$ has at most $r-1$ neighbors in $A_4$. Since $\mathcal{N}(x) \subset S(x,1)$, we have $ B:=S(x,1) \setminus A_4 $ of size $n-r+1$. 

Now consider the neighborhood of the set $B$. Let $y \in B$ and then $y$ has at most $r-1$ neighbors in $A_3$. Since $y$ has exactly $n-1$ neighbors in $S(x,2)$, we have 
$$|\mathcal{N}(y) \cap (S(x,2) \setminus A_3)| \geq (n-1)-(r-1)=n-r.$$

Now let us count the number of edges between $B$ and $S(x,2) \setminus A_3$, denoted by $e(B,S(x,2) \setminus A_3)$. We have
$$e(B,S(x,2) \setminus A_3) \geq |B|(n-r).$$
Note that every vertex in $S(x,2)$ has at most $2$ neighbors in $B$. Define $E:=S(x,2) \setminus A_3$ and we have 
$$2|E| \geq e(B,S(x,2) \setminus A_3).$$ 

Now let us count the number of edges between $E$ and $S(x,3) \setminus A_2$, denoted by $e(E,S(x,3) \setminus A_2)$. Every vertex in $E$ has at most $r-1$ neighbors in $A_2$. Since every vertex in $E$ has exactly $n-2$ neighbors in $S(x,3)$, we have 
$$e(E,S(x,3) \setminus A_2) \geq |E|(n-r-1).$$
Note that every vertex in $S(x,3)$ has at most $3$ neighbors in $E$. Therefore, define $F:= S(x,3) \setminus A_2$ such that 
$$3|F| \geq e(E,S(x,3) \setminus A_2).$$ Therefore, 
$$|F| \geq \frac{(n-r-1)(n-r)(n-r+1)}{6}.$$

Then by Lemma ~\ref{lemma: graph partition for hypercube}, we have a partition for $F$ such that 
$$F= \cup_{i=1}^m F_i$$
where $m \leq 2n^2$ and for every $j \in [m]$, $d_H(a,b) \geq 6$ with $a,b \in F_j$. Note that $\text{support}(a) \cap \text{support}(b)=\emptyset$ for all $a, b \in F_j$. Thus $|F_j| \leq \frac{n}{3}$ for all $j \in [m]$. 

There exists an $i \in [m]$ such that $|F_i| \geq cn$ where $c$ is a positive constant. Since the event $x \in A_2$ is independent of the event that $y \in A_2$ for any $x,y \in F_i$. Therefore, we have 
\begin{align*}
\mathbb{P}(x \notin A_5) & \leq \sum_{j \in [m]} \sum_{\substack{F \subset F_j \\|F|=cn}}\mathbb{P}(F \cap A_2=\emptyset)\\
& \leq 2n^2\binom{\frac{n}{3}}{cn}(n^{-\gamma})^{cn}\\
&=\exp(-cn \log n). \qedhere
\end{align*}
\end{proof}

Now we are ready to prove the upper bound of Theorem ~\ref{main theorem}.
$$\mathbb{P}(A_0 \; \text{does not percolate})\leq \mathbb{P}(\exists x \notin A_5) \leq \sum_{x \in V(Q_n)}\exp(-cn \log n)=o(1).$$

\section{Lower Bound}
In this section we will prove the lower bound in Theorem~\ref{main theorem} for the initial infection probability $p'=\frac{r}{n}-\left(\sqrt{2-a}+\epsilon \right)\frac{\sqrt{r \log n}}{n} $. Following the ideas from \cite{balogh_bollobas_morris_2009} let us introduce a new process closely related to the bootstrap percolation process. The only difference is a less strict requirement for a vertex to be infected at step $i$ for $i \in [k]$ where $k >\frac{2}{a}$. The main new idea is to let the process go for a constant number of steps depending on the value of $a$. 

Now let us formally define this new process. Let $B_i$ denote the set of infected vertices up to step $i$. We define this process as follows: 

* Initially, each vertex is infected with probability $p'$, independent of other vertices.

* If a vertex becomes infected, then it remains infected forever.

*A healthy vertex will be infected at step 1 if it has at least $\rho_1$ infected neighbors at the initial step, i.e., $x \in B_0$ if $x \in B_0$ or $|\mathcal{N}(x) \cap B_0| \geq \rho_1$.

*For $i \in [k]$, a healthy vertex will be infected at step $i$ if it has at least $\rho_i$ infected neighbors at the first step, i.e., $x \in B_i$ if $x \in B_{i-1}$ or $|\mathcal{N}(x) \cap B_{i-1}| \geq \rho_i$.

* Let $i \geq k+1$. A healthy vertex will be infected at step $i$ if it has at least $r$ infected neighbors at step $i$, i.e., $x \in B_{i}$ if $x \in B_{i-1}$ or $|\mathcal{N}(x) \cap B_{i-1}| \geq r$.

Now we need to properly define $\rho_i$ for $i \in [k]$.  Recall $ r=f(n)n^a=n^{a+o(1)}$. Let $p'=\frac{r-\lambda'}{n}$ and $ 
\lambda'=C'\sqrt{r\log n},$ where $C'=\sqrt{2-a}+\epsilon$. Set $\delta:=\frac{C'^2}{2}-\left(1-\frac a2\right)>0$.

Choose a fixed integer $k$ such that $k>\frac{2}{a}$ and let $K$ be a constant, chosen sufficiently large which may depend on $a$. Let $d_2=\left\lfloor\frac{\lambda'}{\log n}\right\rfloor$ and $
d_3=\cdots=d_{k+1}=\lfloor q\rfloor,$
where $q= \left( \frac{3^kk!K\,n}{d_2\log n} \right)^{1/(k-1)}$.

\begin{comment}
Now
\[
d_2=n^{a/2+o(1)}
\]
and
\[
q=n^{b+o(1)},\qquad
b=\frac{1-a/2}{k-1}.
\]
Since \(k>2/a\),
\[
\frac{1-a/2}{k-1}<\frac a2.
\]
Thus
\[
q=o(n^{a/2})=o(\lambda),
\]
and also
\[
d_2=o(\lambda).
\]
\end{comment}

Define $ D':=\sum_{i=2}^{k+1}d_i,$ and then let 
$$
\rho_1=r-D'
$$
and
$$
\rho_j=\rho_{j-1}+d_j,\qquad 2\le j\le k+1.
$$
Note that $\rho_{k+1}=r.$
\begin{lemma} \label{lemma: B_k B_k+1}
  For this new bootstrap percolation process, we have 
  $$\mathbb{P}(B_k=B_{k+1})=1-o(1).$$
\end{lemma}
\begin{proof}
We now outline the main ideas of the proof. Suppose that $x\in B_{k+1}\setminus B_k$. The gaps between the infection thresholds during the first $k$ steps allow us to construct a chain of sets $T_\ell\subseteq S(x,\ell)\cap (B_{k-\ell+1}\setminus B_{k-\ell})$, for $\ell\in[k]$. Moreover, the sizes of these sets can be carefully controlled. At the final level $\ell=k$, we consider the graph $G$ induced by the vertices in $T_k\cup(\mathcal{N}(T_k)\cap S(x,k+1))$. Since $T_k\subseteq S(x,k)\cap(B_1\setminus B_0)$, the infection rule ensures that the vertices in $T_k$ have sufficiently many initially infected neighbors in $S(x,k+1)$. We therefore introduce a random variable counting the edges of $G$ incident to initially infected vertices in $S(x,k+1)$. By estimating its expectation and variance and applying Bernstein's inequality, we obtain a sufficiently strong exponential tail bound to overcome the required union bounds.

Write $G_j=B_j\setminus B_{j-1}.$ If
$$
y\in G_j,\qquad j\ge2,
$$
then
$$
|\mathcal{N}(y)\cap B_{j-1}|\ge\rho_j. 
$$
We also have , because $y\notin B_{j-1}$,
$$
|\mathcal{N}(y)\cap B_{j-2}|\le\rho_{j-1}-1.
$$
Therefore
\begin{equation}
|\mathcal{N}(y)\cap G_{j-1}|\ge d_j+1. \label{eq:d_j+1}
\end{equation}

We need the following claim.  
\begin{claim} \label{claim: witness}
Suppose $x\in G_{m+1},1\le m\le k.$ Starting with $T_0^m=\{x\}$, define
$$
s_0^m=1,\qquad
s_\ell^m=
\left\lfloor
\frac{d_{m+2-\ell}}{3\ell}s_{\ell-1}^m
\right\rfloor ,
\qquad 1\le\ell\le m.
$$
Then there exist
$$
T_\ell^m\subseteq S(x,\ell)\cap G_{m+1-\ell},
\qquad |T_\ell^m|=s_\ell^m.
$$
\end{claim}
\begin{proof}
Indeed, suppose $T_{\ell-1}^m$ has been constructed. Every $ y\in T_{\ell-1}^m
$ has at least $ d_{m+2-\ell}+1$ neighbors in $G_{m+1-\ell}$ by (\ref{eq:d_j+1}).
At most $\ell-1$ of those are inward neighbors relative to $x$. Hence, since all $d_i\to\infty$, it has at least $\frac12d_{m+2-\ell}$
outward neighbors in $S(x,\ell)\cap G_{m+1-\ell}$.

Every vertex of $S(x,\ell)$ has exactly $\ell$ neighbors in $S(x,\ell-1)$. Hence the union of those outward neighbors has size at least
$$
\frac{d_{m+2-\ell}}{2\ell}|T_{\ell-1}^m|
\ge s_\ell^m.
$$
Thus the recursion is valid.
\end{proof}
Fix an integer $m$ between $1$ and $k$. Let $x \in G_{m+1}$ and $T^m_{0}=\{x\}$. Then we define a depth-$m$ witness chain as a sequence of sets $T_0^{(m)},T_1^{(m)},\ldots,T_m^{(m)}$, where $T_j^{(m)}$ lies in $S(x,j) \cap G_{m+1-j}$ and has size $s_j^{(m)}$ for $j \in [m]$. 

Because $m\le k$ is fixed and every $d_i\to\infty$, floor errors are negligible and so
$$
s_{\ell}^m
=
(1+o(1))
\frac{d_{m+1}d_m\cdots d_{m+2-\ell}}{3^{\ell} \ell!}.
$$

Fix $1 \leq m \leq k$ and $T\subset S(x,m)$ such that $|T|=s= s_{m}^m.$ Note that for $m \leq k$ and $n$ large enough, we have 
$$s_m^m \leq s_k^k=(K+o(1))\frac n{\log n}.$$

If $T\subset G_1$, each $y\in T$ has at least
$$
r-D'
$$
initially infected neighbors.
It has exactly $m$ neighbors in $S(x,m-1)$, so it has at least
$$
r-D'-m
$$
infected neighbors in $S(x,m+1)$. 

Now consider a vertex $z\in S(x,m+1)$, define
$$
d_T(z):=|\mathcal{N}(z)\cap T|
$$
and
$$
Y_T:=
\sum_{z\in S(x,m+1)}
d_T(z)\mathbb{I}_{\{z\in A_0\}}.
$$
Then
\begin{equation}
T\subset G_1
\quad\Longrightarrow\quad
Y_T\ge(r-D'-m)s. \label{eq: implies}
\end{equation}

Now in order to apply the Bernstein's inequality we need to calculate $\mathbb{E}(Y_T)$ and $\text{Var}(Y_T)$.

Since $ \sum_zd_T(z)=(n-m)s,$ we have 
$$
\mu':=\mathbb EY_T=p'(n-m)s.
$$
The required deviation is therefore
$$
u:=(r-D'-m-p'(n-m))s.$$
Note that $q=n^{\frac{1-a/2}{k-1}+o(1)}$. Since $k >\frac{2}{a}$, we have $\frac{1-a/2}{k-1}< \frac{a}{2}$ and thus $q=o(\lambda')$. Since $np'=r-\lambda'$ and $D'=o(\lambda')$, $u=(1-o(1))\lambda' s$.

Now let us calculate $\text{Var}(Y_T)$. We have 
$$V:=\text{Var}(Y_T)=np(1-p)\sum_{z \in S(x,m+1)}d_T(z)^2.$$
We will write $d_T(z)^2=d_T(z)+2 \binom{d_T(z)}{2}$. Thus 
$$\sum_{z}d_T(z)^2=\sum_{z}d_T(z)+\sum_{z}2 \binom{d_T(z)}{2}.$$

Since every pair of vertices in $S(x,m)$ can only share at most one common neighbor in $S(x,m+1)$, we have 
$$\sum_{z \in S(x,m+1)}\binom{d_T(z)}{2} \leq \binom{|T|}{2}.$$
Thus 
$$
\sum_{z \in S(x,m+1)}d_T(z)^2 \leq (n-m)s+2\binom{s}{2}.
$$

Since $ s\le (K+o(1))\frac n{\log n}$, we have uniformly over every possible $T$,
$$\sum_zd_T(z)^2=(1+o(1))ns.$$

Therefore
$$V
\le(1+o(1))p'ns
=(1+o(1))rs.$$

Apply Bernstein's inequality in Lemma \ref{lemma: Bernstein}  for the independent random variables $d_T(z)(\mathbb{I}_{\{z \in A_0\}}-p')$ for $z \in S(x,m+1)$. Note that 
$|d_T(z)(\mathbb{I}_{\{z \in A_0\}}-p')| \leq \binom{m+1}{m} \leq k+1$

Thus Bernstein's inequality gives
$$
\mathbb{P}(Y_T-\mu'\ge u)
\le
\exp\left\{-\frac{u^2} {2(V+(k+1)u/3)}
\right\}.
$$
Furthermore,
$$
\frac{u}{V}
=
O\left(\frac{\lambda'}{r}\right)
=
O\left(\sqrt{\frac{\log n}{r}}\right)
=o(1).
$$
By \ref{eq: implies}, we have
$$
\mathbb{P}(T\subset G_1)
\le
\exp\left[
-\left(\frac{C'^2}{2}-o(1)\right)s\log n
\right].
$$

Given $T_{\ell-1}^m$, every possible vertex of $T_\ell^m$ is an outward neighbor of some vertex of $T_{\ell-1}^m$. There are at most $ns_{\ell-1}$ such vertices available for $T_{\ell}^m$.
Therefore the total number $N_m$ of possible depth-m witness chains is at most
$$
N_m\le
\prod_{\ell=1}^m
\binom{ns_{\ell-1}^m}{s_\ell^m}.
$$
At the last step, we have 
$
\frac{s_m^m}{s_{m-1}^m}
=
(1+o(1))\frac{d_2}{3m},
$ and hence
$$
\log\binom{ns_{m-1}^m}{s_m^m}
\le
s_m^m
\left(
\log\frac n{d_2}+O(1)
\right).
$$
Since
$
d_2=\frac{\lambda'}{\log n}(1+o(1))
=n^{a/2+o(1)},
$
we have
$$
\log\frac n{d_2}
=
\left(1-\frac a2+o(1)\right)\log n.
$$
Therefore, the contribution from the last level is at most 
$$\left(1-\frac{a}{2}+o(1)\right)s_m^m \log n.$$

Now let us show that the earlier levels only contribute to the lower-order terms. 
For each $\ell$,
$$
\frac{s_\ell^{m}}{s_{\ell-1}^{m}}
=
(1+o(1))
\frac{d_{m+2-\ell}}{3\ell}.
$$
Every $d_i$ tends to infinity, while $m$ is fixed. Thus each successive witness level is much larger than the previous one, i.e.,
$$
s_1^m
\ll
s_2^m
\ll
\cdots
\ll
s_m^m.
$$
In particular,
$$
s_1^{m}+\cdots+s_{m-1}^{m}
=
o\left(s_m^{m}\right).
$$
For every earlier level, the logarithmic factor
$$
\log\left(
\frac{e n s_{\ell-1}^{(m)}}{s_\ell^{(m)}}
\right)
$$
is at most of order $\log n$. Therefore the total contribution of all levels before the last one is
$$
o\left(s_m^{m}\log n\right).
$$

Consequently,
$$\log N_m
\le
\left(1-\frac a2+o(1)\right)s_m^m\log n.$$

Thus,
\begin{align*}
\mathbb{P}(x\in G_{m+1})
&\le
N_m
\exp\left[
-\left(\frac{C'^2}{2}-o(1)\right)s_m^m\log n
\right]\\
&\le
\exp\left[
-\left(
\frac{C'^2}{2}-1+\frac a2-o(1)
\right)s_m^m\log n
\right].
\end{align*}

Since $\delta=\frac{C'^2}{2}-1+\frac a2>0$,
$$
\mathbb{P}(x\in G_{m+1})
\le
\exp[-(\delta-o(1))s_m^m\log n].
$$
Let $m=k$ and thus we have 
$$
\mathbb{P}(x\in G_{k+1})
\le
\exp[-(\delta K-o(1))n].
$$
By Choosing $K$ large enough, we have 
$$
\mathbb{P}(G_{k+1}\neq \emptyset)
\le
2^n \exp(-(\delta K-o(1))n)
=o(1). \qedhere
$$
\end{proof}

Now we are ready to prove the lower bound of Theorem ~\ref{main theorem}. 
\begin{proof}
If we can show that $B_k \neq V(Q_n)$ with high probability, then combining with lemma ~\ref{lemma: B_k B_k+1} completes the proof

Let $x \in V(Q_n)$. Then we have $\mathbb{P}(x\in B_0)=p=o(1)$,
$$\mathbb{P}(x\in G_1)
\le n^{-C'^2/2+o(1)}=o(1),$$
and 
$$
\mathbb{P}(x\in G_j)=o(1),\qquad 2\le j\le k.
$$
Hence
$$
\mathbb{P}(x\in B_k)\leq \mathbb{P}(x \in B_0)+\sum_{i=1}^k \mathbb{P}(x \in G_i)=o(1).
$$
Therefore, 
$$
\mathbb{P}(B_k=V(Q_n))\le\mathbb{P}(x\in B_k)=o(1). \qedhere
$$
\end{proof}

\section{Discussion and Open Problems}
Question 1. What is the third-order term of $p_c(Q_n,r)$, where $r=f(n)n^a$ and $f(n)=n^{o(1)}$ for $0<a<1$ ?

In \cite{balogh_bollobas_morris_2009}, the third-order term of the critical probability was determined up to a constant factor. It is therefore natural to ask whether a similar result can be established in our setting. At present, however, we do not even have a conjecture regarding the form of the third-order term, although the heuristics discussed in the introduction may still provide some insight. Determining this term would likely require substantial refinements of several arguments developed in this paper.

Question 2. What is $p_c(Q_n,r)$ for other ranges of $r$?

When $r=2$, this problem was first solved by Balogh and Bollob{\'a}s in \cite{balogh_bollobas_2006}, and their results were subsequently refined by Balogh, Bollob{\'a}s, and Morris in \cite{balogh_bollobas_morris_2010}. However, the problem remains largely open for other regimes of $r$. In particular, the case where $r \geq 3$ is a fixed constant is especially interesting, as it has remained open for approximately two decades, largely due to the lack of a sufficiently precise characterization of percolating sets.

\section*{Acknowledgments}
The author would like to thank Prof.~Alexander Barg for suggesting this problem and providing valuable guidance. Additionally, the author extends thanks to Prof.~J{\'o}zsef Balogh for answering questions related to \cite{balogh_bollobas_morris_2009}, Yihan Zhang for insightful discussions, and Anna Geisler for pointing out \cite{collares2024universalbehaviourmajoritybootstrap}.

\section*{AI Disclosure}
ChatGPT 5.6 Plus was used in the preparation of this paper. In particular, during discussions about possible candidates for the second-order term of the critical probability, ChatGPT assisted the author in conjecturing its correct form. Additionally, in the proof of Claim ~\ref{claim: Cov}, ChatGPT suggested to apply McKay's estimate in Lemma ~\ref{Lemma:McKay} to establish the desired result. All other mathematical arguments and claims in this paper were developed by the author, who bears full responsibility for the correctness of the results.
\bibliographystyle{plain}
\bibliography{main.bib}
\end{document}